\documentclass{amsart}

 \newtheorem{thm}{Theorem}[section]
 
 \newtheorem{lem}[thm]{Lemma}
 \newtheorem{prop}[thm]{Proposition}
 \newtheorem{defn}[thm]{Definition}
 
 \newtheorem{ex}{Example}
 
 \usepackage{amsmath,amssymb,amsthm,amscd,amsfonts}
\usepackage{mathrsfs}
\usepackage{latexsym}
\usepackage[all]{xy}
\usepackage{eucal}
\usepackage{verbatim}
\usepackage{color}

\newcommand{\N}{\mathbb{N}}
\newcommand{\Z}{\mathbb{Z}}
\newcommand{\Q}{\mathbb{Q}}
\newcommand{\F}{\mathbb{F}} 
  
\newcommand{\sri}{\twoheadrightarrow}
\newcommand{\iri}{\hookrightarrow}
\renewcommand{\l}{\ell}

\renewcommand{\L}{\Lambda}

\newcommand{\g}{\gamma}
\newcommand{\dl}[1]{\lim_{\buildrel \longrightarrow\over{#1}}}
\newcommand{\il}[1]{\lim_{\buildrel \longleftarrow\over{#1}}}
\newcommand{\ov}{\overline}
\newcommand{\wt}{\widetilde}
\renewcommand{\leq}{\leqslant}
\renewcommand{\geq}{\geqslant}

\DeclareMathOperator{\Coind}{Coind}
\DeclareMathOperator{\Ind}{Ind}
\DeclareMathOperator{\Hom}{Hom}
\DeclareMathOperator{\Gal}{Gal}
\DeclareMathOperator{\cd}{cd}

\input cyracc.def
\font\tencyr=wncyr10
scaled \magstephalf
\def\cyr{\tencyr\cyracc}
\newcommand{\ts}{\mbox{\cyr Sh}}

\begin{document}

\title[On Euler characteristic]{On Euler characteristics of Selmer groups for abelian varieties over global
function fields}

\author{Maria Valentino}
\address{Scuola Normale Superiore di Pisa\\
Piazza dei Cavalieri, 7\\
56126 Pisa}

\email{maria.valentino84@gmail.com}

\thanks{The author is supported by the ERC-Grant n$^\circ$ 267273.}



\begin{abstract}
Let $F$ be a global function field of characteristic $p>0$, $K/F$ an $\l$-adic Lie extension ($\l\neq p$) and $A/F$ an abelian variety.
We provide Euler characteristic formulas for the $\Gal(K/F)$-module $Sel_A(K)_\l$.
\end{abstract}

\maketitle

\section{Introduction}
Let $\l\in \Z$ be  a prime and let $G$ be a profinite $\l$-adic Lie group of finite dimension $d\geq 1$.
Let $M$ be a $G$-module and consider the following properties\begin{itemize}
\item[{\bf 1.}] $H^i(G,M)$ is finite for any $i\geq 0$;
\item[{\bf 2.}] $H^i(G,M)=0$ for all but finitely many $i$.
\end{itemize}

\begin{defn}\label{EuCharDef}
{\em If a $G$-module $M$ verifies {\bf 1} and {\bf 2}, we define the} Euler cha\-rac\-te\-ri\-stic {\em of $M$ as
\[ \chi(G,M):=\prod_{i\geq 0} |H^i(G,M)|^{(-1)^i} \ . \]}
\end{defn}

Let $A/F$ be an abelian variety and, for any extension $L/F$ let $Sel_A(L)_\l$ be the $\l$-power part of the Selmer
group of $A$ over $L$ (for a precise definition see Section \ref{SettNot}). When $G$ is the Galois group of a field extension $L/F$, 
the study of the Euler characteristic $\chi(G,Sel_A(L)_\l)$ is a first step towards understanding the 
relation, predicted by the Iwasawa Main Conjecture, between a characteristic element for $Sel_A(L)_\l$ 
and a suitable $\l$-adic $L$-function. 

In the number field setting, various papers have considered Euler cha\-rac\-te\-ri\-stic formulas for the Selmer groups. Among 
these works, we mention that of Coates and Howson \cite{CH}, where they considered the extension generated by the 
$p$-power torsion points of an elliptic curve $\mathcal{E}/F$. Interesting generalizations can be found in the papers of Van Order \cite{VO} 
and Zerbes \cite{Z2} and \cite{Z1}.

The aim of this paper is to provide Euler characteristic formulas for the Selmer group in the function field case.
Let $F$ be a global function field of characteristic $p>0$ and let $K/F$ be an $\l$-adic Lie 
extension ($\l\neq p$) with Galois group $G$ and unramified outside a finite and nonempty set $S$ of primes of $F$. The case 
$\l=p$, which requires a few more technical tools related to flat cohomology, will be treated in a different paper \cite{BV3}.

Here is a summary of the present work. In Section \ref{AbVarl},
we provide two formulations for $\chi(G,Sel_A(K)_\l)$. The first (see Theorem \ref{ThmEuCh})
mainly depends on the cohomology and Euler characteristic of torsion points while the second (see Theorem \ref{ThmEuCh2}) 
involves more directly
the Tate-Shafarevich group $\ts(A/F)$ (hinting at a connection with values of $L$-functions and the Birch and Swinnerton-Dyer conjecture).
In Section \ref{EllCur}, we specialize our formulas to the case of elliptic curves (all notation are
canonical and will be explained in the next section or soon as they appear) and obtain (see Theorem \ref{EuChEllCur}) 
\begin{thm}
Assume $Sel_{\mathcal{E}}(F)_\l$ is finite, 
$H^2\!(F_S/K,\mathcal{E}[\l^\infty\!])\!\!=\!\!0$, $\chi(G,\mathcal{E}(K)[\l^\infty\!])$ and $\chi(G_v,\mathcal{E}(K_w)[\l^\infty])$ are well defined for any $w|v\in S$
and the map $\psi_K$ in the sequence
\begin{equation}\label{EqSel} Sel_{\mathcal{E}}(K)_\l \iri H^1(F_S/K,\mathcal{E}[\l^\infty]) \stackrel{\psi_K}{-\!\!-\!\!\!\longrightarrow}
\prod_{v\in S} \Coind_G^{G_v} H^1(K_w,\mathcal{E}[\l^\infty]) \end{equation}
is surjective. If $\l\geq 5$ we have
\begin{equation}\label{EqCharE}
\begin{array}{ll} \!\!\chi(G,Sel_{\mathcal{E}}(K)_\l) \!\!= \!\!& \displaystyle{\!\! \!\!\chi(G,\mathcal{E}(K)[\l^\infty]) \frac{|\ts(\mathcal{E}/F)[\l^\infty]|}{|\mathcal{E}(F)[\l^\infty]|^2} }
\cdot \displaystyle{\!\!\!\! \!\!\!\!\prod_{\begin{subarray}{c} v\in S \\ \cd_\l(K_w)=1 \end{subarray}} 
\!\!\!\!\!\!\frac{|\mathcal{E}(F_v)[\l^\infty]|}{\chi(G_v,\mathcal{E}(K_w)[\l^\infty])} }\cdot \\ [6ex]
 & \cdot \displaystyle{\!\! \prod_{\begin{subarray}{c} v\in S\\ \cd_\l(K_w)=0\\ v\,{\rm good\,red.} \end{subarray}} 
\!\!\!|L_v(\mathcal{E},1)|_\l \cdot \!\!\!\prod_{\begin{subarray}{c} v\in S\\ \cd_\l(K_w)=0\\ v\,{\rm split\, mult.\,red.}\end{subarray}} 
\!\!\!\!\!\!\left|\frac{L_v({\mathcal{E}},1)}{c_v}\right|_\l }\end{array}
\end{equation}
\end{thm}

\subsection{Setting and notations}\label{SettNot}
Before moving on we recall the main objects and the setting we shall work with.

Let $F$ be a global function field of characteristic $p>0$ with finite constant field $\F_q$ ($q=p^r$ for 
some $r\in\N$).
For any place $v$ of $F$, let $F_v$ be
the completion of $F$ at $v$ and $\F_v$ its residue field. 
For any Galois extension $L/F$ and any place $v$ of $F$, fix a place $w$ of $L$ lying above $v$ and let $G_v:=\Gal(L_w/F_v)$ be the associated
decomposition group.
We also fix an embedding
of $\ov{F}$ (a separable algebraic closure of $F$) into $\ov{F_v}$ in order to get a restriction map 
$G_{F_v}:=\Gal(\ov{F_v}/F_v) \iri
\Gal(\ov{F}/F)=:G_F\,$.\\
Let $K/F$ be an $\l$-adic Lie extension unramified outside a finite and nonempty set of places 
$S$ of $F$.
We write $G:=\Gal(K/F)$ and assume it has finite dimension $d$ (as $\l$-adic Lie group) and no elements of order $\l$
(then the $\l$-cohomological dimension of $G$  is $\cd_\l(G)=d$ by (\cite[Corollaire (1) p. 413]{Se2}).\\
Let $F_S$ be the maximal (separable) extension of $F$ unramified outside $S$,
so that $K\subseteq F_S\,$. Recall that 
$\cd_\l(\Gal(F_S/F))=2$ (\cite[Corollary 10.1.3 (iii)]{NSW}).

Let $A/F$ be an abelian variety of dimension $g$. We denote by $A^t$ its dual
abelian variety and, as usual, $A[\l^n]$ will be the scheme of $\l^n$-torsion points of $A$, with 
$A[\l^\infty]:=\displaystyle{\dl{n} A[\l^n]}=\bigcup A[\l^n]$. 

For any $\l$-adic Lie group $G$ we denote by
\[ \L(G) = \Z_\l[[G]] := \il{U} \Z_\l [G/U] \]
the associated {\em Iwasawa algebra}, where the limit is taken on the open
normal subgroups of $G$.\\
It is well known that in our setting $\L(G)$ is a Noetherian and
(if $G$ is pro-$\l$ and has no elements of order $\l$) integral domain. 

\noindent Let $H$ be a closed subgroup of $G$. For every $\L(H)$-module
$N$ we consider the $\L(G)$-modules
\[ \Coind^H_G(N):={\rm Map}_{\L(H)}(\L(G), N)\quad {\rm and}\quad
\Ind_H^G(N):=\L(G)\otimes_{\L(H)} N\ . \footnote{Some texts, e.g. \cite{NSW}, switch the definitions of $\Ind^H_G(N)$ and $\Coind^H_G(N)$.} \]

\noindent For a $\L(G)$-module $M$, we denote by $M^{\vee}:=\Hom_{cont}
(M,\Q_\l/\Z_\l)$ its Pon\-trja\-gin dual, which has a natural structure of $\L(G)$-module.

We enlarge our set $S$ so that it contains all primes ramified in $K/F$ and all places of bad reduction for $A$.
Then, the extension $F(A[\l^\infty])/F$ is contained in $F_S/F$ and $A[\l^\infty]$ is an unramified
$G_{F_v}$-module for every $v\notin S$. 
For any Galois extension $L/F$ such that $L\subseteq F_S$, let us consider the map
\begin{align*}
\rho: H^1(F^s/L,A[\l^\infty]) \to \prod_{v\notin S} \Coind^{G_v}_{G} H^1(F_v^s/L_w,A[\l^\infty])\ .
\end{align*} Direct computations on local Galois cohomology groups  give 
(for more details see \cite[Proposition 1.4.4]{V}) $Ker(\rho)\simeq H^1(F_S/L,A[\l^\infty])$ .\\
Let us consider the map:
\[ \eta: H^1(F^s/L,A[\l^\infty]) \to \displaystyle{
\prod_{v\in S} \Coind^{G_v}_G H^1(F_v^s/L_w, A[\l^\infty])}\,.\]
\noindent Thanks to the fact that for $\l\neq p$ the image of the Kummer maps is trivial
(see \cite[Proposition 3.3]{BL}), we have 
\[ Sel_A(L)_\l= Ker(\rho)\cap Ker(\eta) \simeq H^1(F_S/L,A[\l^\infty]) \cap Ker(\eta)\,.\]
Then, we can use the following definition (already employed in \cite{BV2}) for the
Selmer group.

\begin{defn}\label{DefSelAlt} {\em For any finite extension $L$ of $F$, the $\l$-part of
the} Selmer group {\em of $A$ over $L$ is
\[Sel_A(L)_\l = Ker \left\{ H^1(F_S/L, A[\l^\infty]) \to
\prod_{v\in S} \Coind^{G_v}_G H^1(L_w, A[\l^\infty]) \right\}  \]
where $w$ is a fixed place of $L$ lying above $v$.\\
The {\em Tate-Shafarevich group} $\ts(A/L)$
is the group that fits into the exact sequence
\[ A(L)\otimes \Q_\l / \Z_\l \iri Sel_A(L)_\l \sri \ts(A/L)[\l^\infty]\ .\]}
\end{defn}

\noindent For infinite extensions we define the Selmer groups by
taking direct limits on the finite subextensions.
In particular, $Sel_A(K)_\l$ is a $\L(G)$-module whose structure has been studied in \cite{BV1}.\\
If $L/F$ is a finite extension the group $Sel_A(L)_\l$ is a cofinitely generated $\Z_\l$-module
(see, e.g. \cite[III.8 and III.9]{Mi1}). 
Whenever we assume that $Sel_A(L)_\l$ is finite, we have that the $\Z$-rank of
$A(L)$ is 0, hence
\begin{equation}\label{EqSha}
A(L)\otimes \Q_\l/\Z_\l = 0\quad \textrm{and}\quad |Sel_A(L)_\l|= |\ts(A/L)[\l^\infty]|\ .
\end{equation}

\section{Euler Characteristic for abelian varieties}\label{AbVarl}
\subsection{Cohomological lemmas}\label{CohoLeml}
We list some useful results on the cohomology of the $\l$-power torsion points.

\begin{lem}\label{PropH2Zero}
Let $L$ be a finite extension of $F$ contained in $F_S\,$. If $Sel_{A^t}(L)_\l$ is finite,
then \[H^2(F_S/L, A[\l^\infty])=0\ .\]
\end{lem}

\begin{proof}
See (the proof of) \cite[Proposition 4.4]{BV2}.
\end{proof}

\begin{lem}\label{IsoGloAndLoc}
If $Sel_{A^t}(F)_\l$ is finite and $H^2(F_S/K,A[\l^\infty])=0$ we have
\begin{itemize}
\item[{\bf 1.}] $H^i(G,H^1(F_S/K,A[\l^\infty]))\simeq H^{i+2}(G, A(K)[\l^\infty])=0$ $\forall\ i\geq 1$.
\end{itemize}
Moreover, let $w$ be any prime of $K$ such that $w|v\in S$. Then
\begin{itemize}
\item[{\bf 2.}] $H^i(G_v,H^1(K_w,A[\l^\infty]))\simeq H^{i+2}(G_v, A(K_w)[\l^\infty])=0$ $\forall\ i\geq 1$.
\end{itemize}
\end{lem}

\begin{proof}{\bf 1.}
By \cite[Corollary 10.1.3 (iii) and Proposition 3.3.5]{NSW} we have that $\cd_\l(\Gal(F_S/K))\leqslant 2$.
Therefore, our hypothesis on the cohomology group $H^2(F_S/K,A[\l^\infty])$ yields
\[ H^i(F_S/K, A[\l^\infty])=0\quad \forall\ i\geq 2 \ .\]
So, from the Hochschild-Serre spectral sequence, we have
\[\begin{xy}
(-34,0)*+{H^1(G,A(K)[\l^\infty])}="v1";(0,0)*+{H^1(F_S/F,A[\l^\infty])}="v2"; (42,0)*+{H^0(G,H^1(F_S/K,A[\l^\infty]))\to}="v3";
(-34,-7)*+{H^2(G,A(K)[\l^\infty])}="v4";(0,-7)*+{H^2(F_S/F,A[\l^\infty])}="v5"; (42,-7)*+{H^1(G,H^1(F_S/K,A[\l^\infty]))\to}="v6";
(-36,-12)*+{\dots};(0,-12)*+{\dots};(44,-12)*+{\dots};
(-34,-17)*+{H^i(G,A(K)[\l^\infty])}="v7";(-1,-17)*+{H^i(F_S/F,A[\l^\infty])}="v8"; (43,-17)*+{H^{i-1}(G,H^1(F_S/K,A[\l^\infty]))\!\!\to\!\!\dots}="v9";
{\ar@{^{(}->} "v1";"v2"};{\ar@{->} "v2";"v3"};{\ar@{->} "v4";"v5"};{\ar@{->} "v5";"v6"};
{\ar@{->} "v7";"v8"};{\ar@{->} "v8";"v9"};
\end{xy}\](see \cite[Lemma 2.1.3]{NSW}).
Since $\cd_\l(\Gal(F_S/F))=2$ and, by Lemma \ref{PropH2Zero}, the group $H^2(F_S/F, A[\l^\infty])$ is zero, one gets
\[ H^i(G,H^1(F_S/K,A[\l^\infty]))\simeq H^{i+2}(G, A(K)[\l^\infty])\quad \forall\ i\geq 1\ .\]
Moreover, by \cite[Lemma 2.1.4]{NSW}, we have the following isomorphisms
\[ H^i(G,H^1(F_S/K,A[\l^\infty]))=E^{i1}_2\simeq E^{i+1} = H^{i+1}(F_S/F,A[\l^\infty]) \quad \forall\ i\geq 1\ . \]
Then
\[ H^i(G,H^1(F_S/K,A[\l^\infty]))=0\quad \forall\ i\geq 1\ .  \]
{\bf 2.} The argument is the same of part {\bf 1}. In order to get $H^2(K_w,A[\l^\infty])=0$, 
just use \cite[Theorem 7.2.6]{NSW} or \cite[Theorem 7.1.8 (i)]{NSW} according to $K_w$ being a local
field or not (so according to $K/F$ being a finite extension or not). 
\end{proof}

\subsection{Euler characteristic I: Selmer groups and torsion points}
Now we give a first formula for the Euler characteristic of $Sel_A(K)_\l$ which relates it to
the (local and global) Euler characteristics of the $\l^\infty$-torsion points in $K$ and $K_w$.

\begin{thm}\label{ThmEuCh}
With notations as above, assume that $Sel_{A}(F)_\l$ and $Sel_{A^t}(F)_\l$ are finite, 
$H^2(F_S/K,A[\l^\infty])=0$, $\chi(G,A(K)[\l^\infty])$ and $\chi(G_v,A(K_w)[\l^\infty])$ are well defined for any $w|v\in S$
and the map $\psi_K$ in the sequence
\begin{equation}\label{EqSel} Sel_A(K)_\l \iri H^1(F_S/K,A[\l^\infty]) \stackrel{\psi_K}{-\!\!-\!\!\!\longrightarrow}
\prod_{v\in S} \Coind_G^{G_v} H^1(K_w,A[\l^\infty]) \end{equation}
is surjective.
Then $H^i(G,Sel_A(K)_\l)=0$ for any $i\geqslant 2$, the Euler cha\-rac\-te\-ri\-stic of $Sel_{A}(K)_\l$ is well defined and
\begin{equation}\label{EqChi}
\begin{array}{ll} \chi(G,Sel_A(K)_\l) = & \displaystyle{ \chi(G,A(K)[\l^\infty])
\frac{|H^1(F_S/F,A[\l^\infty])|}{|A(F)[\l^\infty]|}} \cdot \\
\ & \cdot \displaystyle{ \left(\prod_{v\in S} \chi(G_v,A(K_w)[\l^\infty]) 
\frac{|H^1(F_v,A[\l^\infty])|}{|A(F_v)[\l^\infty]|} \right)^{-1} } \end{array} 
\end{equation}
where, for every $v\in S$, $w$ is a fixed place of $K$ dividing $v$.
\end{thm}

\begin{proof}
Let us consider the sequence \eqref{EqSel} and 
take its cohomology with respect to $G$. Then, using Shapiro's Lemma (\cite[Proposition 1.6.4]{NSW}), Lemma  
 \ref{IsoGloAndLoc} and the finiteness of the cohomological dimension of $G$, we obtain
\begin{equation}\label{fiveterm}
\begin{xy}
(-30,0)*+{Sel_A(K)_\l^G}="v1";(6,0)*+{H^1(F_S/K,A[\l^\infty])^G}="v2";(6,-15)*+{\displaystyle{\prod_{v\in S}}
H^1(K_w,A[\l^\infty])^{G_v}}="v3";
(45,-15)*+{H^1(G,Sel_A(K)_\l)}="v4";
{\ar@{^{(}->} "v1";"v2"};{\ar@{->}^<<<<{\psi_K^G} "v2";"v3"};{\ar@{->>} "v3";"v4"};
\end{xy}\end{equation}
and $H^i(G,Sel_A(K)_\l)=0$ for any $i\geqslant 2$.
Therefore
\begin{equation}\label{EqEuCh1} \chi(G,Sel_A(K)_\l)= \frac{|Sel_A(K)_\l^G|}{|H^1(G,Sel_A(K)_\l)|} =
\frac{|H^1(F_S/K,A[\l^\infty])^G|}{\displaystyle{\prod_{v\in S}} |H^1(K_w,A[\l^\infty])^{G_v}|} \ .\end{equation}
The Hochschild-Serre spectral sequence yields
\[ \begin{xy} (-30,0)*+{H^1(G,A(K)[\l^\infty])}="v1";(6,0)*+{H^1(F_S/F,A[\l^\infty])}="v2";(6,-15)*+{H^1(F_S/K,A[\l^\infty])^G}
="v3";
(45,-15)*+{H^2(G,A(K)[\l^\infty])}="v4";
{\ar@{^{(}->} "v1";"v2"}{\ar@{->} "v2";"v3"};{\ar@{->>} "v3";"v4"};
\end{xy}\]
Recalling that, by Lemma \ref{IsoGloAndLoc}, $H^i(G,A(K)[\l^\infty])=0$ for any $i\geqslant 3$, we have
\[ \begin{array}{lcl} |H^1(F_S/K,A[\l^\infty])^G| & = &
\displaystyle{\frac{|H^1(F_S/F,A[\l^\infty])||H^2(G,A(K)[\l^\infty])|}{|H^1(G,A(K)[\l^\infty])|}}\\
\ & = & \displaystyle{\frac{|H^1(F_S/F,A[\l^\infty])|\chi(G,A(K)[\l^\infty])}{|H^0(G,A(K)[\l^\infty])|}} \\
\ & = & \displaystyle{\frac{|H^1(F_S/F,A[\l^\infty])|\chi(G,A(K)[\l^\infty])}{|A(F)[\l^\infty]|}}\ . \end{array} \]
The local computations are similar and, substituting in \eqref{EqEuCh1}, we get \eqref{EqChi}.
\end{proof}

\begin{ex} 
{\em Let us consider $K=F(A[\l^\infty])$. By \cite[Proposition 4.5]{BV2}
we know $H^2(F_S/K, A[\l^\infty])=0$.
Moreover, it is easy to see that $\chi(G,A[\l^\infty])$ is well defined thanks to \cite[Th\'eor\`eme 2]{Se3}.
Besides, if all primes in $S$ are of split multiplicative reduction, then the Mumford parametrization and 
the form of the $\l$-power torsion points yields 
$\cd_\l(K_w)=0$. So, the Hochschild-Serre
spectral sequence provides isomorphisms
\[ H^n(G_v,A(K_w)[\l^\infty])\simeq H^n(F_v,A[\l^\infty])\quad\forall\ n\geq 0 \ . \]
Then \[ \chi(G_v,A[\l^\infty]) = \frac{|A[\l^\infty](F_v)|}{|H^1(F_v, A[\l^\infty])|}\] is well defined
thanks to \cite[Theorem 7.1.8]{NSW} and the fact that Tate local duality
(\cite[Theorem 7.2.6]{NSW}) yields $H^2(F_v,A[\l^\infty])=  0$. 
It follows that if both $Sel_A(F)_\l$ and $Sel_{A^t}(F)_\l$ are finite and the map $\psi_K$ in \eqref{EqSel} 
is surjective
\[ \chi(G,Sel_A(K)_\l)=\chi(G,A[\l^\infty])\frac{|H^1(F_S/F,A[\l^\infty])|}{|A(F)[\l^\infty]|} \ .\] }
\end{ex}

\subsection{Euler Characteristic II: descent diagrams}
Now we look for a slightly different formulation for the global factor, which
is more closely related, especially in the case of elliptic curves, to special values of $L$-functions. 
We consider the classical descent diagram 
\begin{equation}\label{Diag1}
\xymatrix{ Sel_A(F)_\l \ar[d]^\alpha \ar@{^(->}[r] & H^1(F_S/F,A[\l^\infty]) \ar@{->>}[r]^{\qquad\psi_F} \ar[d]^\beta &
Im(\psi_F) \ar[d]^{\beta'} \\
Sel_A(K)_\l^G \ar@{^(->}[r] & H^1(F_S/K,A[\l^\infty])^G \ar@{->>}[r]^{\qquad\psi_K^G} &
Im(\psi_K^G) }
\end{equation}
where $Im(\psi_F)$ and $Im(\psi_K^G)$ lie in the diagram
\begin{equation}\label{Diag2}
\xymatrix { Im(\psi_F) \ar[d]^{\beta'} \ar@{^(->}[r] & \displaystyle{\prod_{v\in S} H^1(F_v,A[\l^\infty]) } \ar@{->>}[r] \ar[d]^{\gamma=\oplus \gamma_v} &
Coker(\psi_F) \ar[d]^{\gamma} \\
Im(\psi_K^G) \ar@{^(->}[r] & \displaystyle{ \prod_{v\in S} \Coind_G^{G_v} H^1(K_w,A[\l^\infty])^G } \ar@{->>}[r] &
Coker(\psi_K^G)  \ .}
\end{equation}

\begin{prop}\label{Gamma}
Assume $Sel_{A^t}(F)_\l$ is finite. Then, 
\[ Ker(\beta)=H^1(G,A(K)[\l^\infty])\ \  \mathrm{and}\ \  Coker(\beta)=H^2(G,A(K)[\l^\infty])\,.\] 
Moreover,
\begin{align*}   & Ker(\g)=\!\!\!\prod_{\begin{subarray}{c} v\in S\\ \cd_\l(K_w)=1\end{subarray}} \!\!\! H^1(G_v,A(K_w)[\l^\infty]) \cdot \!\!\! \prod_{\begin{subarray}{c} v\in S\\ \cd_\l(K_w)=0\end{subarray}} \!\!\!H^1(F_v,A[\l^\infty])\\
 & Coker(\g)=\!\!\!\prod_{\begin{subarray}{c} v\in S\\ \cd_\l(K_w)=1\end{subarray}} \!\!\!H^2(G_v,A(K_w)[\l^\infty])\,. 
\end{align*}
\end{prop}

\begin{proof}
For the map $\beta$ just use the Hochschild-Serre five term exact sequence and Lemma \ref{PropH2Zero}.\\
For the map $\g$, by Shapiro's Lemma we can rewrite every $\g_v$ as
\[ \g_v: H^1(F_v,A[\l^\infty]) \to H^1(K_w, A[\l^\infty])^{G_v} \]
for a fixed place $w$ of $K$ dividing $v$. Using again the five term exact sequence and the fact that
$H^2(F_v, A[\l^\infty])=0$, one has
\[ Ker(\g_v)\simeq H^1(G_v,A(K_w)[\l^\infty]) \quad \textrm{and} \quad
Coker(\g_v) \simeq H^2(G_v,A(K_w)[\l^\infty])\ .\] 
If $v$ is totally split we have that $G_v=0$. In this case $\gamma_v$ is an isomorphism. If $v$ is inert 
or ramified, by \cite[Theorem 7.1.8 (i)]{NSW} $\cd_\l(K_w)\leq 1$. This implies that when $\cd_\l(K_w)=0$, 
$\gamma_v$ is the zero-map and we have  $Ker(\gamma_v)\simeq H^1(F_v,A[\l^\infty])$ and $Coker(\gamma_v)=0$.
The claim follows.
\end{proof}

\begin{thm}\label{ThmEuCh2}
With hypotheses as in Theorem \ref{ThmEuCh} one has
\begin{align*} \chi(G,Sel_A(K)_\l)= & \ \chi(G,A(K)[\l^\infty]) \cdot  \frac{|\ts(A/F)[\l^\infty]|}{|A(F)[\l^\infty]||A^t(F)[\l^\infty]|}\cdot \\
 & \cdot \prod_{\begin{subarray}{c} v\in S\\ \cd_\l(K_w)=1\end{subarray}}\!\!\! \!\!\!\frac{|A(F_v)[\l^\infty]|}{\chi(G_v,A(K_w)[\l^\infty])} \cdot 
 \prod_{\begin{subarray}{c} v\in S\\ \cd_\l(K_w)=0\end{subarray}} \!\!\!\!\!\!|H^1(F_v,A[\l^\infty])|\,. 
\end{align*}
\end{thm}

\begin{proof}
Since we are assuming that $Sel_{A^t}(F)_\l$ is finite and $\psi_K$ is surjective, then
equation \eqref{fiveterm} shows that $Coker(\psi_K^G)\simeq H^1(G,Sel_A(K)_\l)$. Therefore
\[ \chi(G,Sel_A(K)_\l)=\frac{|Sel_A(K)_\l^G|}{|Coker(\psi_K^G)|}\ .\]
The snake lemma sequence of diagram \eqref{Diag2} yields
\[|Coker(\psi_K^G)|\!=\! \frac{|Coker(\g')||Coker(\psi_F)|}{|Ker(\gamma')|}\!=\! \frac{|Coker(\psi_F)||Ker(\beta')||Coker(\gamma)|}{|Coker(\beta')||Ker(\gamma)|}\ .\]
Using the snake lemma sequence of diagram \eqref{Diag1} one gets
\[ \begin{array}{lcl} \displaystyle{\frac{|Ker(\beta')|}{|Coker(\beta')|} }& = & \displaystyle{\frac{|Ker(\beta)|}{|Coker(\beta)|}
\cdot\frac{|Coker(\alpha)|}{|Ker(\alpha)|} } \\
\ & = & \displaystyle{ \frac{|H^1(G,A(K)[\l^\infty])|}{|H^2(G,A(K)[\l^\infty])|}\cdot\frac{|Sel_A(K)_\l^G|}{|Sel_A(F)_\l|} }\\
\ & = & \displaystyle{ \chi(G,A(K)[\l^\infty])^{-1} \frac{|A(F)[\l^\infty]||Sel_A(K)_\l^G|}{|\ts(A/F)[\l^\infty]|}\ . } \end{array} \]
Substituting in $Coker(\psi_K^G)$ one gets
\[ |Coker(\psi_K^G)|= 
\frac{|Coker(\psi_F)||Sel_A(K)_\l^G||A(F)[\l^\infty]||Coker(\gamma)|}{\chi(G,A(K)[\l^\infty])|\ts(A/F)[\l^\infty]||Ker(\gamma)|} \]
and
\[ \chi(G,Sel_A(K)_\l) = \chi(G,A(K)[\l^\infty]) \frac{|\ts(A/F)[\l^\infty]||Ker(\gamma)|}{|Coker(\psi_F)||A(F)[\l^\infty]||Coker(\gamma)|} \ .\]
The cardinality of $Ker(\g)$ and $Coker(\gamma)$ can be taken from Proposition \ref{Gamma}. 
To conclude observe that $Sel_{A^t}(F)_\l$ finite yields 
\[ Coker(\psi_F) \simeq (A^t(F)^*)^\vee \]
(where the $^*$ denotes the $\l$-adic completion, see, e.g., \cite[Proposition 4.4]{BV2}). Hence, since $A^t(F)$ is finite by hypothesis,
\[ |Coker(\psi_F)| = |(A^t(F)^*)^\vee| = |\il{n} A^t(F)/\l^n | = |A^t(F)[\l^\infty]| \ . \qedhere\]
\end{proof}

\begin{ex}
 {\em Let $K=F(A[\l^\infty])$ as in the example after Theorem \ref{ThmEuCh}. Suppose that all hypotheses of Theorem \ref{ThmEuCh2}
are verified and that all primes in $S$ are of split multiplicative reduction. In this case, the formula for the Euler characteristic 
of $Sel_A(K)_\l$ is the following:
\[ \chi(G,Sel_A(K)_\l)= \chi(G,A[\l^\infty])  \frac{|\ts(A/F)[\l^\infty]|}{|A(F)[\l^\infty]||A^t(F)[\l^\infty]|}
\cdot \prod_{v\in S} |H^1(F_v,A[\l^\infty])|\,. \]
In order to observe that this formula coincides with that of the previous example, just note that we have the following exact sequence
\[  Sel_A(F)_\l \iri H^1(F_S/F,A[\l^\infty]) \stackrel{\psi_F}{\longrightarrow} \prod_{v\in S} H^1(F_v,A[\l^\infty]) \sri Coker(\psi_F)\,\]
with $|Sel_A(F)_\l|=|\ts(A/F)[\l^\infty]|$ and $|Coker(\psi_F)| = |A^t(F)[\l^\infty]| $. }
\end{ex}

\section{Euler Characteristic for elliptic curves}\label{EllCur}
When $A=\mathcal{E}$ is an elliptic curve we can find an explicit connection between Euler characteristic  and
the Hasse-Weil $L$-function. Replacing, if needed, $F$ by a finite extension, we can (and will) assume that 
the places of multiplicative reductions are of split multiplicative reduction.\\
Let $\wt{\mathcal{E}}_v$ be the image of $\mathcal{E}$ under the usual reduction map at any prime $v$.
We denote by $\wt{\mathcal{E}}_{v,ns}$ the group of non singular points of $\wt{\mathcal{E}}_v$. Moreover, we define
two subset of $\mathcal{E}(F_v)$ as follows:
\begin{align*}
\mathcal{E}_0(F_v)=\{P\in \mathcal{E}(F_v) : \wt{P}\in \wt{\mathcal{E}}_{v,ns}(\F_v)\}\ ,\ \mathcal{E}_1(F_v)=\{P\in \mathcal{E}(F_v) : \wt{P}=O\} \ .
\end{align*}
Finally, let $c_v(\mathcal{E})=|\mathcal{E}(F_v)/\mathcal{E}_0(F_v)|$ be the local Tamagawa factor
of $\mathcal{E}$ at $v$ and $L_v(\mathcal{E},s)$ the Euler factor at $v$ of the Hasse-Weil
$L$-function $L(\mathcal{E},s)$.

\begin{prop}\label{EulFac}
The group $H^1(F_v,\mathcal{E}[\l^\infty])$ is finite and has order $\left|\frac{L_v(\mathcal{E},1)}{c_v}\right|_\l$
(where $|\cdot|_\l$ denotes the normalized $\l$-adic absolute value, i.e., with $|\l|_\l=\l^{-1}\,$).
Moreover, if $v\in S$ is of additive reduction and $\l\geq 5$,
then $|H^1(F_v,\mathcal{E}[\l^\infty])|=1$ (in particular $|Ker(\g_v)|=1$ for those primes).
\end{prop}

\begin{proof}
From \cite[Remark I.3.6 ]{Mi1} we have the following isomorphisms:
\[ \mathcal{E}(F_v)^*\simeq H^1(F_v,\mathcal{E}[\l^\infty])^\vee \ , \]
where $\mathcal{E}(F_v)^*\simeq\begin{displaystyle} \il n \mathcal{E}(F_v)/\l^n \mathcal{E}(F_v)\,\end{displaystyle}$.
 Consider the exact sequence
\[  \mathcal{E}_1(F_v) \iri \mathcal{E}_0(F_v) \sri \wt{\mathcal{E}}_{v,ns}(\F_v) \ . \]
Taking inverse limits of appropriate quotients we obtain
\[  \mathcal{E}_1(F_v)^* \iri \mathcal{E}_0(F_v)^* \sri (\wt{\mathcal{E}}_{v,ns}(\F_v))^*=\wt{\mathcal{E}}_{v,ns}(\F_v)[\l^\infty] \ ,\]
because $\wt{\mathcal{E}}_{v,ns}(\F_v)$ is finite. Since $\mathcal{E}_1(F_v)$ has no points of order $\l$ (\cite[Proposition VII.3.1]{Si}), 
the first term in the previous sequence is trivial and
\[  \mathcal{E}_0(F_v)^*\simeq (\wt{\mathcal{E}}_{v,ns}(\F_v))^* \ .\]
By the exact sequence
\[  \mathcal{E}_0(F_v)^* \iri \mathcal{E}(F_v)^* \sri (\mathcal{E}(F_v)/\mathcal{E}_0(F_v))^*  \ ,  \]
we deduce that the order of $\mathcal{E}(F_v)^*$ is the exact power of
$\l$ dividing the factor $c_v(\mathcal{E})|\wt{\mathcal{E}}_{v,ns}(\F_v)|$. By
\cite[Appendix C]{Si}
\[  |\wt{\mathcal{E}}_{v,ns}(\F_v)|=|\F_v|L_v(\mathcal{E},1)^{-1}\ . \]
Since $|\F_v|$ is a power of $p$ the first claim follows.\\ 
If $v\in S$ is of additive reduction, by \cite[VII, Theorem 6.1]{Si} we have that $\mathcal{E}(F_v)/\mathcal{E}_0(F_v)$ has order at most $4$
and (by \cite[VII, Proposition 5.1 (c)]{Si}) $\wt{\mathcal{E}}_{v,ns}(\F_v)$ is a $p$-group. 
Thus, if $\l\geqslant 5$, then $\l$ does not divide $c_v|\wt{\mathcal{E}}_{v,ns}(\F_v)[\l^\infty]|$
and it follows that 
\[  |H^1(F_v,\mathcal{E}[\l^\infty])|= 1\ . \qedhere\] 
\end{proof}

\begin{thm}\label{EuChEllCur}
With hypotheses as in Theorem \ref{ThmEuCh}, if $\l\geq 5$ we have
\begin{equation}\label{EqCharE}
\begin{array}{ll} \!\!\chi(G,Sel_{\mathcal{E}}(K)_\l) \!\!=\!\! & \displaystyle{ \!\!\!\!\chi(G,\mathcal{E}(K)[\l^\infty]) \frac{|\ts(\mathcal{E}/F)[\l^\infty]|}{|\mathcal{E}(F)[\l^\infty]|^2} }
\cdot \displaystyle{ \!\!\!\!\!\!\!\prod_{\begin{subarray}{c} v\in S \\ \cd_\l(K_w)=1 \end{subarray}} 
\!\!\!\!\!\!\!\frac{|\mathcal{E}(F_v)[\l^\infty]|}{\chi(G_v,\mathcal{E}(K_w)[\l^\infty])} }\cdot \\ [6ex]
 & \cdot \displaystyle{\!\! \prod_{\begin{subarray}{c} v\in S\\ \cd_\l(K_w)=0\\ v\,{\rm good\,red.} \end{subarray}} 
\!\!\!|L_v(\mathcal{E},1)|_\l \cdot\!\!\! \prod_{\begin{subarray}{c} v\in S\\ \cd_\l(K_w)=0\\ v\,{\rm split\, mult.\,red.}\end{subarray}} 
\!\!\!\!\left|\frac{L_v(\mathcal{E},1)}{c_v}\right|_\l }\ . \end{array}
\end{equation}
\end{thm}

\begin{proof}
Just adjust the formula of Theorem \ref{ThmEuCh2} using Proposition \ref{EulFac} and recall that an elliptic curve is self dual, i.e.,  
$\mathcal{E}^t=\mathcal{E}$.\end{proof}

\begin{ex}
{\em Suppose that $G\simeq \Z_\l$, i.e., $K/F$ is the arithmetic $\Z_\l$-extension of $F$ 
(the only one available here, see \cite[Proposition 4.3]{BL2}). Since there is no ramification,
all elements of $S$ are of bad reduction and $\cd_\l(K_w)=0$ cannot happen.
By \cite[Theorem 4.2]{BLV} we know that $\mathcal{E}(K)[\l^\infty]$ 
is a finite group, then $|H^0(G,\mathcal{E}(K)[\l^\infty])|=|H^1(G,\mathcal{E}(K)[\l^\infty])|$ and
$|H^i(G,\mathcal{E}(K)[\l^\infty])|=0$ for any $i\geq 2$.
Hence 
\[  \chi(G,\mathcal{E}(K)[\l^\infty])=1 \ , \]
and with hypotheses as in Theorem \ref{ThmEuCh}, we have
\[  \chi(G,Sel_{\mathcal{E}}(K)_\l)= \frac{|\ts(\mathcal{E}/F)[\l^\infty]|}{|\mathcal{E}(F)[\l^\infty]|^2}\cdot
\prod_{\begin{subarray}{c} v\in S\end{subarray}} \frac{|\mathcal{E}(F_v)[\l^\infty]|}{\chi(G_v,\mathcal{E}(K_w)[\l^\infty])}\ . \] }
\end{ex}

\begin{ex}\label{ExEllCurl} 
{\em Let $F$ be of characteristic $p>3$; consider $K=F(\mathcal{E}[\l^\infty])$ with $\l\geq 5$ and assume $Sel_{\mathcal{E}}(F)_\l$ is finite.
In this setting the only primes in $S$ are the ones of bad reduction. Moreover, using the Tate parametrization it is not hard to check 
that primes in $S$ such that $\cd_\l(K_w)=0$ are exactly the split multiplicative reduction places and those in $S$ with 
$\cd_\l(K_w)=1$ the additive reduction ones. Moreover  
\[  \chi(G,\mathcal{E}(K)[\l^\infty])=1  \]
because of \cite[Theorem 1]{CS1}. Since in this case $\psi_K$ is surjective (see \cite[Theorem III.27]{S}), one has
\[\chi(G,Sel_{\mathcal{E}}(K)_\l)=\frac{|\ts(\mathcal{E}/F)[\l^\infty]|}{|\mathcal{E}(F)[\l^\infty]|^2}
\cdot\!\!\! \prod_{\begin{subarray}{c} v\in S\\  v\,{\rm split\, mult.\,red.}\end{subarray}}\!\!\! \left|\frac{L_v(\mathcal{E},1)}{c_v}\right|_\l\ . \]
For more details on this case see the (unpublished) thesis \cite{S}.}
\end{ex}

\subsection*{Acknowledgment}
The author would like to thank A. Bandini for his guidance and many helpful conversations and suggestions.

\end{document}